\documentclass[12pt, reqno, a4paper]{amsart}

\usepackage{ amssymb, amsmath, enumerate, amsfonts, amsthm, mathrsfs, url, bm, mathtools}

\usepackage{xcolor}  	
\usepackage{hyperref}
\hypersetup{
colorlinks,
   linkcolor={cyan!80!black},
   citecolor={cyan!80!black},
 urlcolor={cyan!80!black}
}

\usepackage{color}

\usepackage[margin=1in]{geometry}

\RequirePackage{doi}

\usepackage{amscd}
\usepackage{amsfonts}
\usepackage{float}
\usepackage{color}
\usepackage[
backend=biber,
style=alphabetic,
]{biblatex}
\usepackage{bookmark}

\usepackage{systeme}

\renewbibmacro{in:}{}
\DeclareFieldFormat{title}{#1}

\DeclareFieldFormat[article]{title}{\mkbibemph{#1}}       % Articles
\DeclareFieldFormat[incollection]{title}{\mkbibemph{#1}}  % Incollection chapters
\DeclareFieldFormat[book]{title}{\mkbibemph{#1}}          % Standalone books
\DeclareFieldFormat[incollection]{booktitle}{#1}          % Plain for incollection book titles
\DeclareFieldFormat[article]{journaltitle}{#1}            % Plain for journals

\AtEveryBibitem{%
  \ifentrytype{misc}{\DeclareFieldFormat{title}{\mkbibemph{#1}}}{}}

\DeclareFieldFormat{eprint:eprint}{arXiv:\href{https://arxiv.org/abs/#1}{#1}}

\usepackage{amssymb}
\newtheorem{theorem}{Theorem}[section]
\newtheorem{lemma}{Lemma}[section]

\newtheorem{proposition}{Proposition}[section]

\theoremstyle{definition}
\newtheorem{definition}{Definition}[section]

\newtheorem{conjecture}{Conjecture}[section]

\theoremstyle{remark}
\newtheorem{remark}{Remark}[section]

\numberwithin{equation}{section}

\newcommand{\Mod}[1]{\ (\mathrm{mod}\ #1)}

\renewcommand{\Re}{\mathrm{Re}}

\renewcommand{\leq}{\leqslant}
\renewcommand{\geq}{\geqslant}

\begin{document}

\title[Non-vanishing of Dirichlet $L$-functions with smooth conductors]{Non-vanishing of Dirichlet $L$-functions with smooth conductors}


%    author one information
\author{}
\address{}
\curraddr{}
\email{}
\thanks{}

%    author two information
\author{Sun-Kai Leung}
\address{D\'epartement de math\'ematiques et de statistique\\
Universit\'e de Montr\'eal\\
CP 6128 succ. Centre-Ville\\
Montr\'eal, QC H3C 3J7\\
Canada}
\curraddr{}
\email{sun.kai.leung@umontreal.ca}
\thanks{}

  %  \subjclass is required.

\date{}

\dedicatory{}


%    Abstract is required.
\begin{abstract}
 Given a large, square-free, smooth conductor, we establish the non-vanishing of the central values for at least $35.9\%$ of the primitive Dirichlet $L$-functions. %This is the first instance of surpassing the $40\%$-barrier for some explicit conductors.
\end{abstract}

\maketitle

%\setcounter{tocdepth}{1}
%\tableofcontents

\section{Introduction}

It is widely believed that all primitive Dirichlet $L$-functions do not vanish at the central point, i.e., for any primitive Dirichlet character $\chi,$ we have $L(\frac{1}{2},\chi) \neq 0.$ This was first conjectured by Chowla \cite{MR177943} for primitive real characters. However, establishing such non-vanishing for all characters appears to be out of reach at the moment.

Given a large, square-free, smooth conductor, we prove in this paper that more than $35.9\%$ of the primitive Dirichlet $L$-functions $L(s,\chi)$ do not vanish at the central point $s=\frac{1}{2}$. In particular, this is the first instance where the proportion surpasses 
$35\%$ for some explicit, positive fraction of moduli.

\begin{theorem} \label{thm:main}
There exist positive real numbers $Q, \eta$ such that for any square-free integer $q>Q$ which is $q^{\eta}$-smooth, i.e., without prime factors exceeding $q^{\eta}$, we have $L(\frac{1}{2},\chi) \neq 0$ for at least $35.9\%$ of the primitive characters modulo $q.$
\end{theorem}

%See \cite{MR3917919} for overcoming the $50\%$-barrier by averaging over the moduli.
A positive proportion of the non-vanishing of Dirichlet $L$-functions was first established in the seminal work of Balasubramanian and Murty \cite{MR1191737}, and later improved to $1/3$ by Iwaniec and Sarnak \cite{MR1689553}. Michel and VanderKam \cite{MR1743500} also obtained the same proportion $1/3$ by introducing a symmetric \textit{two-piece mollifier}. More recently, Bui \cite{MR2978845} reached $34.11\%$ using a three-piece mollifier. In fact, assuming the generalized Riemann hypothesis, a proportion of one-half is attainable.

Restricting to prime moduli, Khan, Mili\'cevi\'c and Ngo \cite{MR4363789} achieved a proportion $5/13 \approx 38.46\%$ by unbalancing the two-piece mollifier. This improved upon the previous work \cite{MR3582014}, which applies the theory of trace functions (see \cite{MR4033731} for an exposition). Interestingly, assuming the existence of a strongly exceptional character modulo $D$, \v Cech and Matom\"aki \cite{MR4735967} showed that for any prime $p \in [D^{300}, D^{O(1)}],$ we have $L(\frac{1}{2},\chi) \neq 0$ for almost all Dirichlet characters $\chi$ modulo $p.$

Regarding smooth moduli, there have been many recent applications of the \textit{$q$-van der Corput method}, first introduced by Heath-Brown \cite{MR485727} (see also \cite{MR1084186}). The most exciting application is, undoubtedly, Zhang's breakthrough \cite{MR3171761} on the level of distribution of primes in progressions to smooth moduli, %and \cite{stadlmann2023primesarithmeticprogressionsbounded}
which leads to bounded gaps between primes (see also \cite{MR3294387} and \cite{stadlmann2023primesarithmeticprogressionsbounded}). On the other hand, Irving \cite{MR3384495} and Xi \cite{MR3826483} enhanced the level the distribution of the divisor functions $\tau_2$ and $\tau_3$ in progressions to smooth moduli, respectively. Likewise, Irving \cite{MR3564622} improved the subconvexity bound for Dirichlet $L$-functions with smooth conductors. Eventually, Wu and Xi \cite{MR4355471} developed a theory of \textit{arithmetic exponent pairs}, further refining both of the aforementioned results of Irving.

To prove Theorem \ref{thm:main}, we apply the standard Cauchy--Schwarz argument using a not necessarily balanced or symmetric two-piece mollifier, which boils down to the evaluation of the mollified first and second moments (see Section \ref{sec:mollifier} for details).
%followed by a general bilinear estimate of Kloosterman sums due to Xi (see the proof of Proposition \ref{prop:2nd}).
\\~\\
\noindent\textit{Notation.} We use the standard big $O$ and little $o$ notations as well as the Vinogradov notation $\ll,$ where the implied constants depend only on the subscripted parameters. We denote $e(x):=e^{2\pi i x}$ for $x \in \mathbb{R}.$
We write $\sum_{\chi \Mod{q}}^{*}$ for a sum over primitive characters modulo $q,$ and $\sum_{\chi \Mod{q}}^{+}$ for a sum over even primitive characters modulo $q.$ We also denote by $\varphi^*(q)$ and $\varphi^+(q)$ the number of primitive characters and even primitive characters, respectively.

\section{Unbalanced two-piece mollifier} \label{sec:mollifier}

Throughout the paper, given a sufficiently small $\eta>0,$ let $q$ be a large, square-free, $q^{\eta}$-smooth integer. Then, as usual, we first mollify the central values $L(\frac{1}{2},\chi)$ for all primitive characters modulo $q$ using the Dirichlet polynomials $M(s,\chi) \approx L(s,\chi)^{-1},$ so that the Cauchy--Schwarz inequality can be applied efficiently and yields
\begin{align} \label{eq:cs}
\frac{1}{\varphi^{*}(q)}\sideset{}{^*}\sum_{\substack{\chi \Mod{q}\\ L(\frac{1}{2},\chi) \neq 0}} 1 
\geq \dfrac{\left| \frac{1}{\varphi^{*}(q)} \sum_{\chi \Mod{q}}^* L(\frac{1}{2},\chi) M(\chi) \right|^2}{ \frac{1}{\varphi^{*}(q)} \sum_{\chi \Mod{q}}^* |L(\frac{1}{2},\chi)M(\chi) |^2 },
\end{align}
where we write $M(\chi)$ for $M(\frac{1}{2},\chi)$ to simplify notation.

In this paper, we utilize the unbalanced or asymmetric two-piece mollifier introduced by Khan, Mili\'cevi\'c and Ngo \cite[p. 1605]{MR4363789} of the shape
\begin{align} \label{eq:mollifier}
M(\chi;\theta_1,\theta_2):=
\frac{\theta_1}{\theta_1+\theta_2} \sum_{m_1 \leq q^{\theta_1-\epsilon}} \frac{y_{m_1}\chi(m_1)}{\sqrt{m_1}} 
+ \frac{\theta_2}{\theta_1+\theta_2} \times \frac{\overline{\tau(\chi)}}{\sqrt{q}} \sum_{m_2 \leq q^{\theta_2-\epsilon}} \frac{y_{m_2}\overline{\chi(m_2)}}{\sqrt{m_2}}
\end{align}
for some $\theta_1, \theta_2 \in (0,1),$ where we denote the dampened Möbius function by
\begin{align*}
y_{m_i}:=\mu(m_i) \left(1-\frac{\log m_i}{\log q^{\theta_i-\epsilon}} \right)
\end{align*}
for $i=1,2,$ and the Gauss sum by
\begin{align*}
\tau(\chi):=\sum_{a \Mod{q}} \chi(a)e\left(\frac{a}{q} \right).
\end{align*}
For convenience, we denote the root number by $\epsilon_{\chi}:=\tau(\chi)/\sqrt{q}.$ Note that $|\epsilon_{\chi}|=1.$

\begin{remark}
Such a twisted mollifier is suggested by the approximate functional equation at the central point 
\begin{align*}
L\Bigl(\frac{1}{2},\chi \Bigr) \approx \sum_{n_1 \leq N_1} \frac{\chi(n_1)}{\sqrt{n_1}} + \overline{\epsilon_{\chi}} \sum_{n_2 \leq N_2} \frac{\overline{\chi(n_1)}}{\sqrt{n_1}}
\end{align*}
with $N_1 N_2 = q$ (see \cite[Theorem 5.3]{MR2061214} for instance), so that
\begin{align*}
L\Bigl(\frac{1}{2},\chi \Bigr)^{-1} \approx  \sum_{n_1 \leq N_1} \frac{\mu(n_1)\chi(n_1)}{\sqrt{n_1}} + \overline{\epsilon_{\chi}} \sum_{n_2 \leq N_2} \frac{\mu(n_2)\overline{\chi(n_1)}}{\sqrt{n_1}}.
\end{align*}
\end{remark}

To evaluate the mollified first and second moments in (\ref{eq:cs}), even and odd primitive characters have to be treated separately. By symmetry, we only consider the even primitive characters $\chi$ here, i.e., $\chi(-1)=1.$ Note that $\varphi^{+}(q)=\frac{1}{2}\varphi^{*}(q)+O(1).$ 

\begin{proposition}[Mollified first moment] \label{prop:1st}
Let $\theta_1, \theta_2 \in (0,1].$ Then as $q \to \infty,$ we have
\begin{align*}
 \sideset{}{^+}\sum_{\chi \Mod{q}} L \Bigl(\frac{1}{2},\chi \Bigr) M(\chi;\theta_1,\theta_2)
 = (1+o(1)) \varphi^{+}(q).
\end{align*}
\end{proposition}

To state the next proposition on the second moment, we first recall the definition of \textit{arithmetic exponent pairs} (see \cite{MR4355471} for further details).

\begin{definition}[Arithmetic exponent pair {\cite[Definition 3.3]{MR4355471}}]
\label{def:pair}
Given a sufficiently small $\eta>0,$ let $q$ be a large, square-free, $q^{\eta}$-smooth number. Also, let $K_q$ be an $\infty$-amiable trace function modulo $q,$\footnote{See \cite[Section 2]{MR4355471} for the definition.} and $W_{\delta}$ be a $\delta$-periodic function. Then we call the triple $(\kappa, \lambda,\nu)$ an \textit{arithmetic exponent pair} if 
\begin{align*}
\mathfrak{S}(K,W;I):=\sum_{n \in I} K_q(n) W_{\delta}(n) \ll_{\mathfrak{c}, \epsilon, \eta} q^{\epsilon} \|W_{\delta}\|_{\infty} 
\left( \frac{q}{|I|} \right)^{\kappa} |I|^{\lambda} \delta^v,
\end{align*}
provided that $|I|<q\delta$.

\end{definition}

\begin{proposition}[Mollified second moment] \label{prop:2nd}
Let $\theta_1, \theta_2 \in (0,1/2]$ satisfying
%\begin{align} \label{eq:system}
%\systeme{
%\frac{\lambda-\kappa}{2} \theta_1 + \theta_2 \leq  \frac{1}{2}-\frac{\kappa}{2},
%\theta_1+\theta_2 \leq  \frac{3}{4},
%2 \theta_1 + \theta_2 \leq  1} \quad{\text{or}} \quad
%\systeme{
% \theta_1 + \frac{\lambda-\kappa}{2} \theta_2 \leq  \frac{1}{2}-\frac{\kappa}{2},
%\theta_1+\theta_2 \leq  \frac{3}{4},
% \theta_1 + 2 \theta_2 \leq  1} 
%\end{align}
\begin{align*}
\frac{1}{2}  < \theta_1 + \theta_2 < \frac{1-\kappa}{1-\kappa+\lambda}
\end{align*}
for some arithmetic exponent pair $(\kappa,\lambda, \nu)$ with $0 \leq \nu-\kappa <1.$
%\footnote{Given a sufficiently small $\eta>0,$ let $q$ be a square-free $q^{\eta}$-smooth number. Also, let $K_q$ be an $\infty$-amiable trace function modulo $q$ and $W_{\delta}$ be a $\delta$-periodic function. Then we call the triple $(\kappa, \lambda,\nu)$ an arithmetic exponent pair if 
%\begin{align*}
%\mathfrak{S}(K,W;I):=\sum_{n \in I} K_q(n) W_{\delta}(n) \ll_{\mathfrak{c}, %\epsilon, \eta} q^{\epsilon} \|W_{\delta}\|_{\infty} 
%\left( \frac{q}{|I|} \right)^{\kappa} |I|^{\lambda} \delta^v,
%\end{align*}
%provided that $|I|<q\delta$ (see \cite[Definition 3.3]{MR4355471}).
%}
Then as $q \to \infty,$ we have
\begin{align*}
 \sideset{}{^+}\sum_{\chi \Mod{q}} \left|L \Bigl(\frac{1}{2},\chi \Bigr)M(\chi;\theta_1,\theta_2) \right|^2 = \left(1+\frac{1}{\theta_1+\theta_2}+o(1) \right) \varphi^+(q).
\end{align*}
\end{proposition}

The essence of our work lies in the mollified second moment, in which an application of the Poisson summation formula transforms the error term into a bilinear form of Kloosterman sums, allowing the $q$-van der Corput AB-processes to be executed more efficiently. Otherwise, one might not be able to surpass Bui's record of $34.11\%$, which holds for general moduli (see Section \ref{sec:mollmom} for details).
%This differs from most previous work and enables us to break the $40\%$-barrier.
%which would be impossible otherwise (see Section \ref{sec:mollmom} for details).

%For prime moduli, Khan, Mili\'cevi\'c and Ngo \cite{MR4363789} take $(\theta_1, \theta_2)=(3/8, 1/4).$ 
%For smooth moduli, 
%Then, one can verify that $(\theta_1, \theta_2) \approx (0.3255, 0.3489)$ maximizes $\theta_1+\theta_2$ while satisfying (\ref{eq:system}). By symmetry, the pair $(\theta_1, \theta_2) \approx (0.3489, 0.3255)$ also works, i.e., the second piece mollifies as efficiently as the first.

%\begin{remark}
%It is tempting to repeat the sequence of $q$-van der Corput AB-processes $\cdots A^2BA^2BA^2BA^2BA^2B(0,1).$ However, none of these produce a larger $\theta_1+\theta_2.$
%\end{remark}

\begin{proof}[Proof of Theorem \ref{thm:main} assuming Propositions \ref{prop:1st} and \ref{prop:2nd}]

Applying Proposition \ref{prop:1st} and Proposition \ref{prop:2nd} (similarly for odd primitive characters) to (\ref{eq:cs}), we have
\begin{align*}
\frac{1}{\varphi^{*}(q)}\sideset{}{^*}\sum_{\substack{\chi \Mod{q}\\ L(\frac{1}{2},\chi) \neq 0}} 1
\geq (1-o(1))\frac{\theta_1+\theta_2}{1+\theta_1+\theta_2}
\end{align*}
for those $(\theta_1, \theta_2) \in (0,1/2]^2$ satisfying $\frac{1}{2}<\theta_1+\theta_2< \frac{1-\kappa}{1-\kappa+\lambda}$.
Then, the theorem follows by taking the arithmetic exponent pair $
B A B A^2 B A^3 B \left(0, 1, 0 \right)=
\left( \frac{52}{243}, \frac{50}{81}, \frac{202}{243} \right),$  
where 
\begin{align*}
A(\kappa, \lambda, \nu):= \left( \frac{\kappa}{2(\kappa+1)} , \frac{\kappa+\lambda+1}{2(\kappa+1)}, \frac{\kappa}{\kappa+1} \right) 
\end{align*}
is the A-process and
\begin{align*}
B(\kappa, \lambda, \nu):=\left( \lambda-\frac{1}{2}, \kappa+\frac{1}{2}, \lambda-\kappa+\nu \right)
\end{align*}
is the B-process of the $q$-van der Corput method (see \cite[Section 3]{MR4355471}). 
\end{proof}

Therefore, it remains to prove Propositions \ref{prop:1st} and \ref{prop:2nd}, especially the latter. 

\section{Mollified moments} \label{sec:mollmom}

We begin with the decomposition
\begin{align} \label{eq:decomp}
M(\chi;\theta_1, \theta_2)=\frac{\theta_1}{\theta_1+\theta_2}M_{IS}(\chi;\theta_1)
+\frac{\theta_2}{\theta_1+\theta_2}M_{MV}(\chi;\theta_2),
\end{align}
where
\begin{align*}
M_{IS}(\chi;\theta_1):=\sum_{m_1 \leq q^{\theta_1-\epsilon}} \frac{y_{m_1}\chi(m_1)}{\sqrt{m_1}} 
\end{align*}
and
\begin{align*}
M_{MV}(\chi;\theta_2):=
\overline{\epsilon_{\chi}} \sum_{m_2 \leq q^{\theta_2-\epsilon}} \frac{y_{m_2}\overline{\chi(m_2)}}{\sqrt{m_2}}.
\end{align*}
%with $M_i=q^{\theta_i-\epsilon}$ for $i=1,2.$

%Before embarking on the computation of the mollified second moment, we require the following lemmas.

The computation of the mollified first moment is now considered standard.

\begin{proof}[Proof of Proposition \ref{prop:1st}]
%This is essentially \cite[Section 3]{MR1743500}. See also \cite[Section 3.1]{MR4363789}.
Using the decomposition (\ref{eq:decomp}), it boils down to
\begin{align*}
\sideset{}{^+}\sum_{\chi \Mod{q}} L\Bigl( \frac{1}{2}, \chi \Bigr) M_{IS}(\chi;\theta_1)
\end{align*}
and
\begin{align*}
\sideset{}{^+}\sum_{\chi \Mod{q}} L\Bigl( \frac{1}{2}, \chi \Bigr) M_{MV}(\chi;\theta_2).
\end{align*}
In fact, both the first and the second sum are $(1+o(1))\varphi^+(q)$ (see \cite[p. 947]{MR1689553} and \cite[Section 3]{MR1743500}, respectively). Therefore, the proposition follows.
\end{proof}

To compute the mollified second moment, we first isolate the main term.

\begin{lemma} \label{lem:2nd}
Let $\theta_1, \theta_2 \in (0,1/2].$ Then
\begin{gather*}
\sideset{}{^+}\sum_{\chi \Mod{q}} 
\left|  L\Bigl( \frac{1}{2}, \chi \Bigr) M(\chi;\theta_1, \theta_2) \right|^2
=\left(1+\frac{1}{\theta_1+\theta_2}+o(1) \right) \varphi^+(q)  
+ \Sigma^{\flat}_{IS, MV},
\end{gather*}
where $\Sigma^{\flat}_{IS, MV}$
\begin{align*}
:=
\Re  \frac{2}{\sqrt{q}} 
\sum_{\substack{q=cd\\(c,d)=1}}
 %\mu^2(r) 
 \varphi(c) 
\underset{m_1m_2n_1 \neq 1}{\underset{\substack{m_1 \leq q^{\theta_1-\epsilon}\\ m_2 \leq q^{\theta_2-\epsilon}\\(m_1m_2, q)=1}}{\sum \sum} 
\underset{\substack{n_1 \geq 1\\ n_2 \geq 1 \\(n_1n_2,q)=1 }}{\sum \sum}}
\frac{y_{m_1} y_{m_2} }{\sqrt{m_1 m_2}} 
\frac{1}{\sqrt{n_1n_2}} V \left( \frac{n_1n_2}{q} \right) 
e\left( \frac{\overline{dm_1m_2n_1} n_2}{c} \right)
\end{align*}
for some fixed nonnegative function $V(x) \ll_A (1+x)^{-A}.$
\end{lemma}

\begin{proof}
This is essentially \cite[Section 4]{MR3917919}. For the sake of completeness, we sketch a proof here. Using the decomposition (\ref{eq:decomp}), we have 
\begin{gather*}
|M(\chi;\theta_1, \theta_2)|^2 =\left(\frac{\theta_1}{\theta_1+\theta_2}\right)^2 | M_{IS}(\chi;\theta_1) |^2+ \left(\frac{\theta_2}{\theta_1+\theta_2}\right)^2
| M_{MV}(\chi;\theta_2) |^2 \\
+ \frac{2 \theta_1 \theta_2}{(\theta_1+\theta_2)^2}\Re ( M_{IS}(\chi;\theta_1) \overline{M_{MV}(\chi;\theta_2)} ) . 
\end{gather*}
Arguing as in \cite[Section 2.3]{MR2978845}, we have
\begin{align*}
\Sigma_{IS}:=\sideset{}{^+}\sum_{\chi \Mod{q}} \left| L\Bigl( \frac{1}{2}, \chi \Bigr)  M_{IS}(\chi;\theta_1) \right|^2 = \left(1+\frac{1}{\theta_1}+o(1) \right)\varphi^+(q) .
\end{align*}
Since $|\epsilon_{\chi}|=1,$ we also have
\begin{align*}
\Sigma_{MV}:=\sideset{}{^+}\sum_{\chi \Mod{q}} \left| L\Bigl( \frac{1}{2}, \chi \Bigr)  M_{MV}(\chi;\theta_2) \right|^2 = \left(1+\frac{1}{\theta_2}+o(1) \right)\varphi^+(q) .
\end{align*}
%\begin{align*}
%\sideset{}{^+}\sum_{\chi \Mod{q}}  L\Bigl( \frac{1}{2}, \chi \Bigr) | M_{MV}%(\chi;\theta_2) |^2
%\end{align*}
Therefore, it remains to evaluate
\begin{align*}
\Sigma_{IS, MV}:=\sideset{}{^+}\sum_{\chi \Mod{q}} \left| L\Bigl( \frac{1}{2}, \chi \Bigr) \right|^2 M_{IS}(\chi;\theta_1) \overline{M_{MV}(\chi;\theta_2)}.
\end{align*}

A standard argument using the smoothed approximate functional equation for $L(s,\chi)$ at $s=\frac{1}{2}$ (see \cite[Section 6.1]{MR1743500} and \cite[Section 4]{MR3917919} for details) yields
\begin{align*}
\Sigma_{IS, MV}
=2 \underset{\substack{m_1 \leq q^{\theta_1-\epsilon}\\ m_2 \leq q^{\theta_2-\epsilon}\\(m_1m_2, q)=1}}{\sum \sum}
\frac{y_{m_1} y_{m_2} }{\sqrt{m_1 m_2}} 
\underset{\substack{n_1 \geq 1 \\ n_2 \geq 1\\ (n_1n_2,q)=1 }}{\sum \sum}
\frac{1}{\sqrt{n_1n_2}} V \left( \frac{n_1n_2}{q} \right)
\sideset{}{^+}\sum_{\chi \Mod{q}} \epsilon_{\chi} \chi(m_1m_2n_1) \overline{\chi(n_2)}
\end{align*}
for some fixed nonnegative function $V(x) \ll_A (1+x)^{-A}.$ By the almost orthogonal relation
\begin{align*}
\sideset{}{^+}\sum_{\chi \Mod{q}} \chi(m) \overline{\chi(n)}
=\frac{1}{2} \sum_{\substack{q=cd\\c | m \pm n}} \mu(d) \varphi(c)
\end{align*}
whenever $(mn,q)=1$ (see \cite[Lemma 4.1]{MR2782664} for instance), one can show that 
\begin{align*}
\sideset{}{^+}\sum_{\chi \Mod{q}}   \epsilon_{\chi}  \chi(m_1m_2n_1) \overline{\chi(n_2)} 
=\Re \,  \frac{1}{\sqrt{q}} \sum_{\substack{q=cd\\(c,d)=1}}  \varphi(c) 
e\left( \frac{\overline{dm_1m_2n_1} n_2}{c} \right)
\end{align*}
by opening the root number $\epsilon_{\chi}$. The main term of $\Sigma_{IS, MV}$ arises from $m_1=m_2=n_1=1,$ which is 
\begin{align*}
\Sigma_{IS, MV}^{\sharp}:=
2 \sideset{}{^+}\sum_{\chi \Mod{q}} \epsilon_{\chi} \sum_{n \geq 1} \frac{\overline{\chi(n)}}{\sqrt{n}}
V\left( \frac{n}{q} \right).
\end{align*}
Then, one can show by using the approximate functional equation (see \cite[Section 4]{MR3917919} for details) that 
\begin{align*}
\Sigma_{IS, MV}^{\sharp}=(1+o(1))\varphi^+(q),
\end{align*}
and the lemma follows
\end{proof}

We also require a general bilinear estimate of Kloosterman sums (trace functions), which is essentially a weighted and also twisted version of \cite[Lemma 2.3]{MR3826483}, to bound the error term transformed by the Poisson summation formula. 

\begin{lemma} \label{lem:bilinear}
Given three positive integers $M, N, \delta$ and two sequences of complex numbers $\boldsymbol{\alpha}=(\alpha_m)_{M<m \leq 2M}, \boldsymbol{\beta}=(\beta_n)_{N<n \leq 2N},$ let $K_q$ be an $\infty$-amiable trace function modulo $q,$ and $W_{\delta}$ be a $\delta$-periodic function with $\|W_{\delta}\|_{\infty} \ll_{\epsilon} {\delta}^{\epsilon}.$ Suppose $(\kappa, \lambda, \nu)$ is an arithmetic exponent pair with $0 \leq \nu-\kappa <1$. Then
\begin{gather*}
\mathfrak{B}(\boldsymbol{\alpha},\boldsymbol{\beta};K,W):=\sum_{\substack{M < m \leq 2M\\(m,q)=1}} \sum_{\substack{N < n \leq 2N\\(n,q)=1}} \alpha_m \beta_n  K_q(\overline{m}n) W_{\delta} (\overline{m}n) \\
\ll_{\mathfrak{c}, \epsilon, \eta}
(q M \delta)^{\epsilon} M^{\frac{1}{2}} N^{\frac{1}{2}} \|\boldsymbol{\alpha}\|_2 \|\boldsymbol{\beta}\|_2
\left( q^{\frac{\kappa}{2}} N^{-\frac{1-\lambda+\kappa}{2}}  \delta^{\frac{\nu}{2}}
+q^{-\frac{1}{4}} + M^{-\frac{1}{2}}
\right).
\end{gather*}

\end{lemma}

Our proof builds on a generalization of the classical Pólya–Vinogradov inequality.

\begin{lemma}[Pólya–Vinogradov inequality {\cite[Proposition 3.4]{MR4355471}}] \label{lem:pv}
The triple $(\frac{1}{2}, \frac{1}{2}, 1)$ is an arithmetic exponent pair (see Definition \ref{def:pair}).
\end{lemma}

%This is essentially \cite[Lemma 2.3]{MR3826483} twisted by a weight function $W_{\delta}.$

\begin{proof}[Proof of Lemma \ref{lem:bilinear}]
For readers unfamiliar with the language of trace functions, it suffices, for the proof of Proposition \ref{prop:2nd}, to assume that $K_q(x) = \operatorname{Kl}_2(ax; q)$ for some $(a, q) = 1,$ where
\begin{align*}
\operatorname{Kl}_2(b;q):=\frac{1}{\sqrt{q}} \sum_{\substack{u \Mod{q}\\(u,q)=1}} e \left( \frac{bu+\overline{u}}{q} \right)
\end{align*}
is the normalized Kloosterman sum.

Applying the Cauchy--Schwarz inequality, we have
\begin{gather}
|\mathfrak{B}(\boldsymbol{\alpha},\boldsymbol{\beta};K,W)|^2 \leq 
\|\boldsymbol{\beta}\|_2^2 \sum_{n} \left| \sum_m \alpha_m K_q(\overline{m}n)W_{\delta}(\overline{m}n) \right|^2 \nonumber \\
\leq  \|\boldsymbol{\beta}\|_2^2 \sum_{m_1} \sum_{m_2} |\alpha_{m_1}| |\alpha_{m_2}|
\left| \sum_{n}  K_q(\overline{m_1}n) \overline{K_q(\overline{m_2}n)} W_{\delta}(\overline{m_1}n) \overline{W_{\delta}(\overline{m_2}n)} \right|. \label{eq:babkw}
\end{gather}
For $m_1=m_2,$ by assumption, the contribution to the right-hand side of \eqref{eq:babkw} is
\begin{align} \label{eq:n1=n2}
\ll_{\epsilon} (q\delta)^{\epsilon} N \|\boldsymbol{\alpha}\|_2^2 \|\boldsymbol{\beta}\|_2^2.
\end{align}
For $m_1 \neq m_2,$ let $d=(m_1-m_2,q).$ Since $q$ is square-free, we have $(d,q/d)=1.$ Then the Chinese remainder theorem yields 
\begin{align*}
K_{q}(\overline{m_1}n) \overline{K_q(\overline{m_2}n)} = |K_{d} (\overline{m_1}n) |^2 K_{q/d}(\overline{m_1}n) \overline{K_{q/d}(\overline{m_2}n)}.
\end{align*}

On one hand, suppose $N \leq q.$ Then by the definition of arithmetic exponent pairs, we have
\begin{align*}
\sum_{n}  K_q(\overline{m_1}n) \overline{K_q(\overline{m_2}n)} W_{\delta}(\overline{m_1}n) \overline{W_{\delta}(\overline{m_2}n)}
\ll_{\mathfrak{c}, \epsilon, \eta} (d N \delta)^{\epsilon} \left( \frac{q}{d N} \right)^{\kappa} 
M^{\lambda} (d\delta)^{\nu},
\end{align*}
so that the contribution to the right-hand side of \eqref{eq:babkw} is
\begin{gather}
\ll_{\mathfrak{c}, \epsilon, \eta}
(N\delta)^{\epsilon}  q^{\kappa} N^{\lambda-\kappa} \delta^{\nu} \|\boldsymbol{\alpha}\|_2^2
\underset{m_1 \neq m_2}{\sum \sum} |\alpha_{m_1}| |\alpha_{m_2}| (m_1-m_2,q)^{\nu-\kappa+\epsilon} \nonumber\\
\ll_{\mathfrak{c}, \epsilon, \eta} 
(N\delta)^{\epsilon} q^{\kappa} N^{\lambda-\kappa}  \delta^{\nu} M \|\boldsymbol{\alpha}\|_2^2 \|\boldsymbol{\beta}\|_2^2
\label{eq:B1}
\end{gather}
by the assumption that $0 \leq \nu-\kappa<1.$

On the other hand, suppose $N>q.$ Then applying Lemma \ref{lem:pv} that
%(see \cite[Proposition 3.4]{MR4355471} for instance) that
\begin{align*}
\mathfrak{S}(K,W;I):=\sum_{n \in I} K_q(n) W_{\delta}(n) \ll_{\mathfrak{c}, \epsilon} (q|I|)^{\epsilon} \| W_{\delta} \|_{\infty} \sqrt{q} \delta,
\end{align*}
 we obtain
 \begin{align*}
 \sum_{n}  K_q(\overline{m_1}n) \overline{K_q(\overline{m_2}n)} W_{\delta}(\overline{m_1}n) \overline{W_{\delta}(\overline{m_2}n)}
\ll_{\mathfrak{c}, \epsilon}
 \left( q d N  \delta \right)^{\epsilon} N \sqrt{d/q},
 \end{align*}
 so that the contribution to the right-hand side of \eqref{eq:babkw} is
\begin{gather}
\ll_{\mathfrak{c}, \epsilon} 
(q N \delta)^{\epsilon} q^{-1/2} N \|\boldsymbol{\beta}\|_2^2 
\underset{m_1 \neq m_2}{\sum \sum} |\alpha_{m_1}| |\alpha_{m_2}| (m_1-m_2,q)^{1/2+\epsilon} \nonumber \\
\ll_{\mathfrak{c}, \epsilon} 
(q N \delta)^{\epsilon} q^{-1/2} M N \|\boldsymbol{\alpha}\|_2^2 \|\boldsymbol{\beta}\|_2^2. \label{eq:B2}
\end{gather}

Therefore, the lemma follows from (\ref{eq:n1=n2}), (\ref{eq:B1}) and (\ref{eq:B2}).
\end{proof}

We are now ready to finish the proof of Proposition \ref{prop:2nd}.

\begin{proof}[Proof of Proposition \ref{prop:2nd}] 
Invoking Lemma \ref{lem:2nd}, it suffices to show that $\Sigma^{\flat}_{IS, MV}\ll_{\epsilon} q^{1-\epsilon}.$ 
Let us introduce a smooth dyadic partition of unity. Let $\Phi : [0, \infty) \to [0, \infty)$ be a fixed smooth function with compact support contained in $[1, 2]$ such that
\begin{align*}
\sum_{j \in \mathbb{Z}} \Phi \left( \frac{x}{2^j} \right) = 1.
\end{align*}
Then, we can write
\begin{align*}
\Sigma^{\flat}_{IS, MV}
&=\underset{\substack{M_1, M_2, N_1, N_2 \text{ dyadic} \\ M_1M_2N_1>1}}{\sum \cdots \sum} \Sigma^{\flat}_{IS, MV}(M_1, M_2, N_1, N_2),
\end{align*}
where $\Sigma^{\flat}_{IS, MV}(M_1, M_2, N_1, N_2)$
\begin{gather} 
:=
\Re  \frac{2}{\sqrt{q}} 
\sum_{\substack{q=cd\\(c,d)=1}}
 %\mu^2(r) 
 \varphi(c) 
\underset{\substack{M_1< m_1 \leq \min\{2M_1, q^{\theta_1-\epsilon}\}\\ M_2 < m_2 \leq \min\{2M_2,q^{\theta_2-\epsilon} \}\\(m_1m_2, q)=1}}{\sum \sum} \frac{y_{m_1} y_{m_2} }{\sqrt{m_1 m_2}}
 \nonumber \\
 \underset{\substack{n_1 \geq 1\\ n_2 \geq 1 \\(n_1n_2,q)=1 }}{\sum \sum}
\frac{1}{\sqrt{n_1n_2}} 
\Phi \left( \frac{n_1}{N_1} \right) \Phi \left( \frac{n_2}{N_2} \right)
V \left( \frac{n_1n_2}{q} \right) 
e\left( \frac{\overline{dm_1m_2n_1} n_2}{c} \right). \label{eq:m1m2n1n2}
\end{gather}

On one hand, arguing as in \cite[Section 6.2]{MR1743500}, one can show that
\begin{align} \label{eq:1stcond}
\Sigma^{\flat}_{IS, MV}(M_1, M_2, N_1, N_2) \ll_{\epsilon} q^{1+\epsilon} \sqrt{\frac{M^{*}_1M^{*}_2N_1}{qN_2}},
\end{align}
where $M^{*}_i:=\min \{ 2M_i, q^{\theta_i-\epsilon} \}$ for $i=1,2.$

On the other hand, suppose $N_1 N_2 > q^{1+\epsilon}$ for some $\epsilon>0.$ Then by the definition of $\Phi$ and $V,$ we obtain
\begin{align*}
\Sigma^{\flat}_{IS, MV}(M_1, M_2, N_1, N_2) \ll_A
q^{\epsilon} \sqrt{qM_1^* M_2^* N_1 N_2} \left( \frac{N_1N_2}{q} \right)^{-A} 
\ll_{\epsilon} q^{-100}.
\end{align*}
Therefore, from now on, we always assume $N_1N_2 \leq q^{1+\epsilon}$ for any $\epsilon>0.$ 
Using the Möbius inversion formula followed by separating $n_1$ into residue classes modulo $c$, we have
\begin{gather*}
\sum_{\substack{n_1 \geq 1\\(n_1,q)=1}}
\frac{1}{\sqrt{n_1}} \Phi \left( \frac{n_1}{N_1} \right) V \left( \frac{n_1n_2}{q} \right)
e\left( \frac{\overline{dm_1m_2n_1} n_2}{c} \right) \\
=
\sum_{e | d} \mu(e) 
\sum_{\substack{n_1 \geq 1\\ e | n_1 \\ (n_1,c)=1}}
\frac{1}{\sqrt{n_1}} \Phi \left( \frac{n_1}{N_1} \right) V \left( \frac{n_1n_2}{q} \right)
e\left( \frac{\overline{dm_1m_2n_1} n_2}{c} \right) \\
= \sum_{e | d} \mu(e) 
\sum_{\substack{a \Mod{c}\\(a,c)=1}} e\left( \frac{\overline{dm_1m_2a} n_2}{c} \right)
\sum_{\substack{n_1 \geq 1\\ e | n_1 \\n_1 \equiv a \Mod{c}}}
\frac{1}{\sqrt{n_1}}\Phi \left( \frac{n_1}{N_1} \right) V \left( \frac{n_1n_2}{q} \right).
\end{gather*}
Letting $n_1=en_1'$ followed by the Poisson summation formula, this becomes
\begin{gather}  
\sum_{e | d} \frac{\mu(e)}{\sqrt{e}}
\sum_{\substack{a \Mod{c}\\(a,c)=1}} e\left( \frac{\overline{dm_1m_2a} n_2}{c} \right)
\sum_{\substack{n_1' \geq 1\\n_1' \equiv a\overline{e} \Mod{c}}}
\frac{1}{\sqrt{n_1'}} \Phi \left( \frac{en_1'}{N_1} \right) V \left( \frac{en_1'n_2}{q} \right) \nonumber \\
= \frac{\sqrt{N_1}}{c} \sum_{e | d} \frac{\mu(e)}{e}
\sum_{\substack{a \Mod{c}\\(a,c)=1}} e\left( \frac{\overline{dm_1m_2a} n_2}{c} \right)
 \sum_{h \in \mathbb{Z}} F(h) e \left( \frac{a\overline{e}h}{c} \right) \nonumber \\
 = \sqrt{\frac{N_1}{c}}\sum_{e | d} \frac{\mu(e)}{e} 
 \sum_{h \in \mathbb{Z}} F(h) \operatorname{Kl_2}(\overline{dem_1m_2}hn_2;c), \label{eq:n1sum} 
\end{gather}
%\begin{align} \label{eq:n1sum}
%&\sum_{\substack{n_1 \geq 1\\(n_1,q)=1}}
%\frac{1}{\sqrt{n_1}} \Phi \left( \frac{n_1}{N_1} \right) V \left( \frac{n_1n_2}{q} \right)
%e\left( \frac{\overline{dm_1m_2n_1} n_2}{c} \right) \nonumber \\
%=&\sum_{\substack{a \Mod{c}\\(a,c)=1}} 
%e\left( \frac{\overline{dm_1m_2a} n_2}{c} \right)
%\sum_{\substack{n_1 \geq 1\\ (n_1,d)=1 \\n_1 \equiv a \Mod{c}}}
%\frac{1}{\sqrt{n_1}}\Phi \left( \frac{n_1}{N_1} \right) V \left( \frac{n_1n_2}{q} \right)
%\nonumber \\
%=& \sum_{\substack{a \Mod{c}\\(a,c)=1}} 
%e\left( \frac{\overline{dm_1m_2a} n_2}{c} \right)
%\sum_{\substack{n_1 \geq 1\\ (n_1,d)=1 \\n_1 \equiv a \Mod{c}}}
%\frac{1}{\sqrt{n_1}}\Phi \left( \frac{n_1}{N_1} \right) V \left( \frac{n_1n_2}{q} \right) %\nonumber \\
%=& \sum_{\substack{a \Mod{c}\\(a,c)=1}} 
%e\left( \frac{\overline{dm_1m_2a} n_2}{c} \right)
%\sum_{h \in \mathbb{Z}} 
%F(h) e\left( \frac{ah}{c} \right) \nonumber \\
%=& N_1\sum_{h \in \mathbb{Z}} 
%F(h) \cdot \frac{S(\overline{dm_1m_2}n_2,h;c)}{c},
%\end{align}
where 
\begin{align*}
F(h):=\int_{-\infty}^{\infty} \frac{1}{\sqrt{x}}
\Phi(x) V \left( \frac{N_1n_2 x}{q} \right) e \left( -\frac{N_1hx}{ce} \right) dx
\end{align*}
is the Fourier transform, and 
\begin{align*}
\operatorname{Kl}_2(a;c):=\frac{1}{\sqrt{c}}\sum_{\substack{x \Mod{c}\\(x,c)=1}} e \left( \frac{ax+\overline{x}}{c} \right)
\end{align*}
is the normalized Kloosterman sum. 
%Applying the Kuznetsov--Selberg identity (see \cite[Section A]{MR3826483} for an elementary proof), we have
%\begin{align*}
%\frac{S(\overline{dem_1m_2}n_2,h;c)}{c}
%=\sum_{f | (h,c)} \frac{f}{c} \cdot S\left( \frac{\overline{dem_1m_2}n_2h}{e^2},1;\frac{c}{e} %\right),
%\end{align*}
%so that \eqref{eq:n1sum} becomes
%\begin{align} \label{eq:tobesub}
%N_1 \sum_{e | c} \frac{e}{c} 
%\sum_{\substack{h \in \mathbb{Z}\\ e | h}}
%F(h) S\left( \frac{\overline{dm_1m_2}n_2h}{e^2},1;\frac{c}{e} \right).
%\end{align}
Note that the interchange of summations is justified since 
$F(h) \ll_A \min\{1, (N_1h/ce)^{-A}\}$ for any $A>0,$ by repeated integration by parts. Moreover, using Weil's bound 
\begin{align} \label{eq:weil}
\operatorname{Kl}_2(a;c) \leq \tau(c)  \sqrt{(a,c)},
\end{align}
one can show that the contribution of $|h|>q^{\epsilon}\frac{ce}{N_1}$ to \eqref{eq:m1m2n1n2} is $\ll_{\epsilon} q^{-100}.$ Also, let $c_1=c/(h,c)$ and $c_2=(h,c).$ Then it follows from the Chinese remainder theorem that
\begin{align*}
\operatorname{Kl_2}(\overline{dem_1m_2}hn_2;c)=
\operatorname{Kl_2}(\overline{c_2^2 d e m_1 m_2}hn_2;c_1)
\operatorname{Kl_2}(\overline{c_1^2 d e m_1 m_2}hn_2;c_2).
\end{align*}
Therefore, inserting \eqref{eq:n1sum} into \eqref{eq:m1m2n1n2} yields
$\Sigma^{\flat}_{IS, MV}(M_1, M_2, N_1, N_2) $
\begin{gather*}
= \Re \, 2 \sqrt{\frac{N_1}{q}}
\sum_{\substack{q=c_1c_2d\\(c_1c_2,d)=1}}
 %\mu^2(r) 
 \frac{\varphi(c_1c_2)}{\sqrt{c_1c_2}}
 \sum_{e | d} \frac{\mu(e)}{e}
 \sum_{\substack{n_2 \geq 1\\(n_2,q)=1}}
 \frac{1}{\sqrt{n_2}}  \Phi \left( \frac{n_2}{N_2} \right)
 \sum_{\substack{|h| \leq q^{\epsilon}\frac{c_1c_2e}{N_1}\\c_2 | h, \, (h,c_1)=1}} F(h) \\
\underset{\substack{M_1< m_1 \leq \min\{2M_1, q^{\theta_1-\epsilon}\}\\ M_2 < m_2 \leq \min\{2M_2,q^{\theta_2-\epsilon} \}\\(m_1m_2, q)=1}}{\sum \sum}
\frac{y_{m_1} y_{m_2} }{\sqrt{m_1 m_2}}\operatorname{Kl_2}(\overline{c_2^2 d e m_1 m_2}hn_2;c_1)
\operatorname{Kl_2}(\overline{c_1^2 d e m_1 m_2}hn_2;c_2)+O_{\epsilon}(q^{-100}).
\end{gather*}
For $h=0,$ the contribution to $\Sigma^{\flat}_{IS, MV}(M_1, M_2, N_1, N_2)$ is
\begin{gather} 
\ll q^{\epsilon}\left( \frac{M_1^*M_2^*N_1}{qN_2} \right)^{\frac{1}{2}}
\sum_{\substack{q=c_1c_2d\\(c_1c_2,d)=1}} (c_1c_2)^{\frac{1}{2}} |\operatorname{Kl}_2(0;c_1)| |\operatorname{Kl}_2(0;c_2)|
\sum_{e | d} \frac{1}{e} \nonumber \\
\ll_{\epsilon} q^{\epsilon} \left(\frac{M_1^*M_2^*N_1}{qN_2}\right)^{\frac{1}{2}}
\label{eq:ramanujan}
\end{gather}
since the Kloosterman sums are reduced to Ramanujan sums. Therefore, this is negligible and it remains to consider $h \neq 0.$

For $c_2> q^{\theta_1+\theta_2-\frac{1}{2}+10\epsilon},$ it follows form Weil's bound (\ref{eq:weil}) that the contribution to $\Sigma^{\flat}_{IS, MV}(M_1, M_2, N_1, N_2) $ is
\begin{gather*}
\ll_{\epsilon} q^{\epsilon}\left( \frac{M_1M_2N_2}{qN_1} \right)^{\frac{1}{2}} \sum_{\substack{q=c_2c_2d\\(c_1c_2,d)=1}} c_1^{\frac{3}{2}} c_2^{\frac{1}{2}} 
\ll_{\epsilon} q^{1+\epsilon} M_1^* M_2^* q^{-\frac{1}{2}} c_2^{-1},
\end{gather*}
which is negligible.

For $c_2 \leq q^{\theta_1+\theta_2-\frac{1}{2}+10\epsilon},$ since $c_1$ is now $c_1^{O(\eta)}$-smooth and $(h,c_1)=1,$ Lemma \ref{lem:bilinear} is applicable and we have
\begin{gather*}
 \sum_{\substack{n_2 \geq 1\\(n_2,q)=1}}
 \frac{1}{\sqrt{n_2}}  \Phi \left( \frac{n_2}{N_2} \right)
 \sum_{\substack{|h| \leq q^{\epsilon}\frac{c_1c_2e}{N_1}\\c_1 | h, \, (h,c_2)=1}} F(h)
\underset{\substack{M_1< m_1 \leq \min\{2M_1, q^{\theta_1-\epsilon}\}\\ M_2 < m_2 \leq \min\{2M_2,q^{\theta_2-\epsilon} \}\\(m_1m_2, q)=1}}{\sum \sum}
\frac{y_{m_1} y_{m_2} }{\sqrt{m_1 m_2}} \\
\operatorname{Kl_2}(\overline{c_2^2 d e m_1 m_2}hn_2;c_1) 
\operatorname{Kl_2}(\overline{c_1^2 d e m_1 m_2}hn_2;c_2) \\
\ll_{\epsilon, \eta} 
q^{\epsilon} \frac{c_1e}{N_1} \sqrt{M_1^* M_2^*N_2} 
\left( c_1^{\frac{\kappa}{2}} \left( \frac{c_1e}{N_1} N_2 \right)^{-\frac{1-\lambda+\kappa}{2}} c_2^{\frac{\nu}{2}}
+c_1^{-\frac{1}{4}} +(M_1^* M_2^*)^{-\frac{1}{2}}
\right),
\end{gather*}
so that the contribution to $\Sigma^{\flat}_{IS, MV}(M_1, M_2, N_1, N_2) $ is
\begin{gather} 
\ll_{\epsilon, \eta} 
q^{\epsilon} \sqrt{\frac{M_1^* M_2^* N_2}{qN_1}}
\sum_{\substack{q=c_1c_2d\\c_1>q^{?}}} c_1^{\frac{3}{2}} c_2^{\frac{1}{2}}
\left(  c_1^{\frac{\kappa}{2}} \left( \frac{c_1}{N_1} N_2 \right)^{-\frac{1-\lambda+\kappa}{2}} c_2^{\frac{\nu}{2}}
+c_1^{-\frac{1}{4}} +(M_1^* M_2^*)^{-\frac{1}{2}} \right) \nonumber\\
\ll_{\epsilon, \eta} 
q^{\epsilon} \sqrt{\frac{M_1^* M_2^* N_2}{qN_1}}
\left(q^{1+\frac{\lambda}{2}} N_1^{\frac{1-\lambda+\kappa}{2}} N_2^{-\frac{1-\lambda+\kappa}{2}}+ q^{\frac{5}{4}} + q^{\frac{3}{2}} (M_1^* M_2^*)^{-\frac{1}{2}} \right), \nonumber
\end{gather}
which is
\begin{gather}
\ll_{\epsilon, \eta} 
q^{1+\epsilon}
\left( q^{-\frac{1-\lambda}{2}} (M_1^* M_2^*)^{\frac{1}{2}} \left( \frac{N_1}{N_2} \right)^{-\frac{\lambda-\kappa}{2}} +q^{-\frac{1}{4}} \left( \frac{M_1^* M_2^* N_2}{N_1} \right)^{\frac{1}{2}} +  \left( \frac{N_1}{N_2} \right)^{-\frac{1}{2}} \right). \label{eq:dualbound}
\end{gather}

Finally, combining (\ref{eq:1stcond}), (\ref{eq:ramanujan}) and (\ref{eq:dualbound}), we conclude that
\begin{gather*}
\Sigma^{\flat}_{IS, MV}(M_1, M_2, N_1, N_2) \ll_{\epsilon,\eta} q^{1+\epsilon}
\min \left\{ \left( \frac{M_1^* M_2^* N_1}{qN_2} \right)^{\frac{1}{2}}, 
\right.
\\
\left. 
 q^{-\frac{1-\lambda}{2}} (M_1^* M_2^*)^{\frac{1}{2}} \left( \frac{N_1}{N_2} \right)^{-\frac{\lambda-\kappa}{2}}
 +q^{-\frac{1}{4}} \left( \frac{M_1^* M_2^* N_2}{N_1} \right)^{\frac{1}{2}} +  \left( \frac{N_1}{N_2} \right)^{-\frac{1}{2}} 
\right\}.
\end{gather*}
Since by assumption $\theta_1+\theta_2<\frac{1-\kappa}{1-\kappa+\lambda},$ this is $\ll_{\epsilon,\eta} q^{1-\epsilon}$ and the proposition follows.
\end{proof}

\section{Limitation}

Unfortunately, our method cannot break the $40\%$-barrier, even under the assumption of square-root cancellation for short sums of trace functions (see Conjecture \ref{conj} below). It is tempting to apply the Poisson summation formula once more in the variable $n_2$. The resulting hyper-Kloosterman sums, however, it seems, do not offer further savings. Nevertheless, it is certainly plausible to surpass $36\%$ ($40\%$ conditionally) by incorporating more pieces into a necessarily unbalanced two-piece mollifier (see \cite{MR2978845} and \cite{čech2025optimalitymollifiers}).

\begin{conjecture}[Arithmetic exponent pair hypothesis] \label{conj}
We have $(0,1/2,0)$ as an arithmetic exponent pair, i.e., given a sufficiently small $\eta>0,$ let $q$ be a large, square-free, $q^{\eta}$-smooth number. Also, let $K_q$ be an $\infty$-amiable trace function modulo $q$ and $W_{\delta}$ be a $\delta$-periodic function. Then
\begin{align*}
\mathfrak{S}(K,W;I):=\sum_{n \in I} K_q(n)W_{\delta}(n) \ll_{\mathfrak{c}, \epsilon, \eta} (q\delta)^{\epsilon}|I|^{\frac{1}{2}+\epsilon},
\end{align*}
provided that $|I|<q\delta.$\footnote{See \cite[p. 59]{MR1297543} for the classical exponent pair hypothesis.}
\end{conjecture}

Assuming the hypothesis. Then, it follows from Proposition \ref{prop:2nd} that one can take any $\theta_1, \theta_2 \in (0,1/2]$ satisfying $\frac{1}{2}<\theta_1 + \theta_2 < \frac{2}{3},$ so that the proportion of non-vanishing is
\begin{align*}
\frac{1}{\varphi^{*}(q)}\sideset{}{^*}\sum_{\substack{\chi \Mod{q}\\ L(\frac{1}{2},\chi) \neq 0}} 1\geq (1-o(1))\frac{\theta_1+\theta_2}{1+\theta_1+\theta_2} \geq \frac{2}{5}-o(1).
\end{align*}
Therefore, it is interesting to see whether the $40\%$ or even the $50\%$-barrier can be overcome by averaging over smooth moduli.\footnote{Unfortunately, errors related to Bui's mollifier have been identified in the published work of Pratt \cite{MR3917919}, which will be corrected in due course (see \cite{pratt} for the notice).} 
%which will be explored in another paper.\footnote{Unfortunately, errors related to Bui's mollifier have been identified in the published work of Pratt \cite{MR3917919}, which will be corrected in due course (see \cite{pratt} for the notice).} 

\section*{Acknowledgements}
The author is grateful to Andrew Granville for his advice and encouragement. He would also like to thank %Maksym Radziwiłł, 
Ping Xi and Fei Xue for helpful discussions, as well as Martin \v Cech, Kaisa Matom\"aki 
and the anonymous referee %and Igor Shparlinski 
for their valuable comments.

%\nocite{*}
\printbibliography
\end{document}